\documentclass[12]{amsart}
\usepackage{amsmath,amssymb,amsthm,color, enumerate}

\newtheorem{thm}{Theorem}[section]
\newtheorem{lem}[thm]{Lemma}

\newtheorem{prop}[thm]{Proposition}

\newtheorem{ex}[thm]{Example}
\theoremstyle{remark}

\theoremstyle{definition}
    \newtheorem{defn}[thm]{Definition}

\newtheorem{rem}[thm]{Remark}

\numberwithin{equation}{section}
\def\F {{\mathbb F}}
\def\R {{\mathbb R}}

\def\C {{\mathbb C}}

\def\B {{\mathcal B}}

\def\L {{\mathcal L}}
\def\M {{\mathcal M}}

\newcommand{\im}{\operatorname{im}}

\def\red#1 {\textcolor{red}{#1 }}
\def\blue#1 {\textcolor{blue}{#1 }}

\numberwithin{equation}{section}

\begin{document}

\title[Enumerating Invariant Subspaces]{Enumerating Invariant Subspaces of $\R^n$}

\author{Joshua Ide}
\address{Department of Mathematics, Shippensburg University, Pennsylvania, USA}
\email[Josh Ide]{ji1574@ship.edu}

\author{Lenny Jones}
\address{Department of Mathematics, Shippensburg University, Pennsylvania, USA}
\email[Lenny~Jones]{lkjone@ship.edu}

\date{\today}

\begin{abstract}
In this article, we develop an algorithm to calculate the set of all integers $m$ for which there exists a linear operator $T$ on $\R^n$ such that $\R^n$ has exactly $m$ $T$-invariant subspaces. A brief discussion is included as how these methods might be extended to vector spaces over arbitrary fields.
\end{abstract}

\subjclass[2010]{Primary 47A15; Secondary 47A46, 05A17}
\keywords{linear operator, vector space, invariant subspace, Jordan canonical form.}
%\subjclass{Primary: 11B37, 11B39; Secondary: 11D59}
%\subjclass[2000]{Primary 12E05; Secondary 11C08, 11B83}
%\keywords{reducible and irreducible polynomials, cyclotomic polynomials}

\maketitle

%{\bf Mathematics Subject Classification:} xxxxx \\

%{\bf Keywords:} xxxxx, xxxxx
\section{Introduction}
The importance of invariant subspaces is undeniable in the study of linear operators on vector spaces. One of the most infamous unsolved problems in mathematics, known as the Invariant Subspace Problem, is the question of whether every separable Hilbert space $V$ of dimension greater than 1 over $\C$ has a closed, nontrivial $T$-invariant subspace $W$ (i.e. $T(W)\subseteq W$) when $T$ is a bounded linear operator on $V$. For a very recent survey on this problem, see \cite{CP2011}. In this article, our main focus is on finite dimensional vector spaces $V\simeq \R^n$ over $\R$. While these spaces are Hilbert spaces, we do not require the added structure of an inner product and the associated topology. In our situation, it is easy to see that there always exist $T$-invariant subspaces of $V$. For example, the kernel of $T$, the image of $T$, the trivial subspace and the space $V$ itself are $T$-invariant. Moreover, any eigenspaces of $T$ are also $T$-invariant. Aside from these obvious examples, there are other well-known techniques for generating $T$-invariant subspaces. For example, for any vector $v\in V$, the cyclic subspace of $V$ spanned by $\{v, T(v), T^2(v), \ldots \}$ is $T$-invariant. Also, any generalized eigenspace for $T$ is $T$-invariant. But even armed with these additional tools, it is not clear in general how to determine if, for given positive integers $m$ and $n$, there exists a linear operator $T$ on $\R^n$ such that $\R^n$ has \emph{exactly} $m$ $T$-invariant subspaces.
  For example, could it be that there exists a linear operator $T$ on $\R^4$ such that $\R^4$ has exactly 10 $T$-invariant subspaces? Indeed, no such linear operator exists. In this article, we develop an algorithm to calculate the set of all integers $m$ for which there exists a linear operator $T$ on $\R^n$ such that $\R^n$ has exactly $m$ $T$-invariant subspaces. We end with a brief discussion as to how these ideas might be extended to vector spaces over arbitrary fields.

\section{Notation and Definitions}
In this section we present some notation and definitions from linear algebra and combinatorics that are associated to the techniques used in this paper.
\subsection{Linear Algebra}
Throughout this article we let $\F$ be a field, and $V$ be an $n$-dimensional vector space over $\F$, so that $V\simeq\F^n$. We let $T\in \L(V)$, where $\L(V)$ is the space of all linear operators on $V$. If $\B$ is an ordered basis for $V$, we use the notation $[T]_{\B}$ to denote the matrix for $T$ relative to the ordered basis $\B$, and simply write $[T]$ for the matrix for $T$ relative to the standard basis. We denote the image of $T$ and the kernel of $T$ as $\im(T)$ and $\ker(T)$, respectively. For a subspace $W$ of $V$, we let $T|_{W}$ indicate the restriction of $T$ to $W$. We let $c_T:=c_T(x)$ and $m_T:=m_T(x)$ denote the characteristic polynomial and minimal polynomial of $T$, respectively.
For the majority of this article, we are concerned with the case when $\F=\R$. Therefore, we let $\{e_i \mid i=1,2,\ldots, n\}$ denote the standard basis for $\R^n$, and %unless stated otherwise,
 we assume that all eigenvalues of $T$, if any, are real numbers.

\begin{defn}[Standard Jordan Block]\label{Def:StandardJordanBlockDef}
Let $\lambda\in \R$, and let $k>0$ be an integer. A \emph{standard Jordan block} of size $k$ corresponding to $\lambda$ is the $k\times k$ matrix
\[J_{\lambda, k}:=\begin{bmatrix} \lambda & 1 & 0 & \cdots & 0 & 0\\ 0 & \lambda & 1 & \cdots&0 & 0\\
 \vdots & \vdots & \vdots & \ddots & \vdots & \vdots\\ 0 & 0 & 0 & \cdots & \lambda & 1\\
 0 & 0 & 0 & \cdots & 0 & \lambda \end{bmatrix}.\]
For any integer $r\ge 0$, we also define
\[E_{\lambda}^{(r)}:=\ker\left( \left(J_{\lambda,k}-\lambda I_{k}\right)^{r}\right).\]
\end{defn}

\begin{defn}[Real Jordan Block]\label{Def:RealJordanBlockDef}
Let $\zeta=a+bi$ with $b>0$, % \in \mathbb{C}\setminus\R$,
 and let $C=\begin{bmatrix} a & -b \\ b & a\end{bmatrix}$. For an integer $k>0$, a \emph{real Jordan block} of size $2k$ corresponding to $\zeta$ is the $2k \times 2k$ matrix
\[J_{\zeta,k}^*:=\begin{bmatrix} C & I_2 & 0 & \cdots & 0 & 0\\ 0 & C & I_2 & \cdots & 0 & 0\\ \vdots & \vdots & \vdots & \ddots & \vdots & \vdots\\ 0 & 0 & 0 & \cdots & C & I_2\\ 0 & 0 & 0 & \cdots & 0 & C\end{bmatrix}.\]
Let \[D_{k}=\begin{bmatrix} C & 0 &  \cdots & 0 & 0\\ 0 & C &  \cdots & 0 & 0\\ \vdots & \vdots & \ddots & \vdots & \vdots \\ 0 & 0 & \cdots & C & 0\\ 0 & 0 & \cdots & 0 & C\end{bmatrix}\] be a $2k\times 2k$ matrix, and for any integer $r\ge 0$, define
\[U_{\zeta}^{(r)}:=\ker\left( \left(J_{\zeta,k}^*-D_{k}\right)^r \right).\]
\end{defn}

\begin{rem}\label{Rem:Kernels}
  Note that $v=(v_1,v_2,v_3,\ldots,v_k)\in E_{\lambda}^{(r)}$ if and only if $v_{r+1}=v_{r+2}=\cdots =v_k=0$. Similarly, $v=(v_1,v_2,v_3,\ldots,v_{2k})\in U_{\zeta}^{(r)}$ if and only if $v_{2r+1}=v_{2r+2}=\cdots =v_{2k}=0$.
\end{rem}

\subsection{Compositions, Partitions and Multipartitions}\label{Sec: CPM}

A \emph{composition} $\mu$ of the integer $n\ge 0$, written $\mu\vDash n$, is a finite ordered set of nonnegative integers $(\mu_1,\mu_2,\ldots ,\mu_r)$ such that 0 appears at most once in the list and $\sum_{i=1}^r\mu_i=n$. %The integers $\mu_i$ are called the \emph{components} of $\mu$.
 The integers $\mu_i$ are referred to as the \emph{parts} of $\mu$. We define the \emph{length} of a composition $\mu=(\mu_1,\dots,\mu_r)$, denoted $\ell(\mu)$, to be $r$. %If $\mu$ is a composition we write $|\mu|=\sum_{i=1}^r\mu_i$.  If $|\mu|=n$ we say $\mu$ is a \emph{composition of $n$}, and we write $\mu\vDash n$.
A \emph{partition} of $n$ is a composition whose parts are all positive and weakly decreasing from left to right.   If $\mu$ is a partition of $n$ we write $\mu\vdash n$.
%A \emph{multipartition} is a finite ordered set of partitions.
A \emph{$\mu$--multipartition of $n$} corresponding to $\mu=(\mu_1,\dots,\mu_r)\vDash n$ is an ordered set of partitions $\theta:=\theta(\mu)=(\theta_1,\dots,\theta_r)$ with $\theta_i\vdash\mu_i$ for each $i=1,\dots, r$. %By the \emph{components} of a $\mu$--multipartition of $n$, we mean the components of its constituent partitions.
Note that from any $\mu$--multipartition $\theta$ of $n$ we can derive a unique composition of $n$ by removing the internal parentheses.
We call this unique composition, which we denote $d(\mu,\theta)$, the \emph{derived composition} of the $\mu$--multipartition $\theta$.
For example, if $\mu=(5,6,3)$ and $\theta =((4,1),(3,2,1),(2,1))$, then $d(\mu,\theta)=(4,1,3,2,1,2,1)$ of $n=14$.

\section{The Results}\label{Sec: Results}

\begin{prop}\label{Prop:Similar}
Let $V$ be a finite dimensional vector space over a field $\F$, and let $S,T\in \L(V)$. If $S$ and $T$ are similar, then $V$ has exactly the same number of $S$-invariant subspaces as $T$-invariant subspaces.
\end{prop}
\begin{proof}
Since $S$ and $T$ are similar, there exists an invertible matrix $P$ such that $P^{-1}SP=T$.
 First suppose $W$ is an $S$-invariant subspace of $V$, and consider the subspace $P^{-1}(W)$ of $V$. Then
\[T(P^{-1}(W))=P^{-1}SP(P^{-1}(W))=P^{-1}(S(W))\subseteq P^{-1}(W).\] On the other hand, if $T(W)\subseteq W$, then $S(P(W))\subseteq P(W)$. This establishes a one-to-one correspondence between the $S$-invariant subspaces and the $T$-invariant subspaces, which completes the proof.
\end{proof}

The existence of a standard Jordan canonical form for a linear operator $T \in \L(\F^n)$, where $\F$ is an algebraically closed field, is well-known. What we require here is a generalization of this fact for linear operators on $\R^n$. In Section \ref{Sec: Extension}, we discuss a generalization over an arbitrary field. However, that particular generalization, when viewed over $\R$, differs from the one we use here. For more information, the interested reader should consult \cite{AW1992}.

\begin{thm}[Real Jordan Canonical Form]\label{Thm:RJCF}
Let $T \in \L(\R^n)$ and suppose that $c_T=f_1\cdots f_r$, where either $f_j(x)=(x-\lambda_j)^{l_j}$ or $f_j(x)=\left( \left(x-a_j\right)^2+b_j^2\right)^{m_j}$ with $b_j>0$. Then there is an ordered basis $\B$ of $\R^n$ such that
\[[T]_{\B}=\begin{bmatrix} J_1 & 0 &  \cdots & 0 & 0\\ 0 & J_2 &  \cdots & 0 & 0\\ \vdots & \vdots & \ddots & \vdots & \vdots \\ 0 & 0 & \cdots & J_{t-1} & 0\\ 0 & 0 & \cdots & 0 & J_t\end{bmatrix}=\bigoplus_{j=1}^t J_j,\]
where either $J_j$ is a standard Jordan block corresponding to the eigenvalue $\lambda_j \in \R$, or $J_j$ is a real Jordan block corresponding to $\zeta_j=a_j+b_ji$. In either case, we refer to $J_j$ as a Jordan block of $T$.
\end{thm}

\begin{rem}\label{Rem:wlogRJCF}
By Proposition \ref{Prop:Similar} and Theorem \ref{Thm:RJCF}, we may assume without loss of generality that $[T]$ is in real Jordan canonical form.
\end{rem}
\begin{thm}\label{Thm:Inf}
Let $T \in \L(\R^n)$, and let $\alpha \in \C$ be such that $c_T(\alpha)=0$, where either $\alpha=\lambda\in \R$ or $\alpha=a+bi$ with $b>0$.
Suppose that $[T]= \bigoplus_{j=1}^t J_j$ is the real Jordan form for $T$. If there exist two or more Jordan blocks for $T$ corresponding to $\alpha$, then $\R^n$ has infinitely many $T$-invariant subspaces.
\end{thm}
\begin{proof}
  Suppose first that $\alpha=\lambda \in \R$, so that $\lambda$ is an eigenvalue of $T$. Let $E_{\lambda}$ be the corresponding eigenspace of $T$. If $T$ has exactly $s>1$ standard Jordan blocks corresponding to $\lambda$, then $\dim(E_{\lambda})=s$. Then, since $s>1$, there are infinitely many 1-dimensional subspaces of $E_{\lambda}$, and all such subspaces are $T$-invariant. Thus, the theorem is proved in this case.

  Now suppose that $\alpha=a+bi$, with $b>0$. Assume that $T$ has exactly $s>1$ real Jordan blocks $J_1,J_2,\ldots , J_s$, of respective sizes $2k_1,2k_2,\ldots , 2k_s$, corresponding to $\alpha$. Let $U$ be the subspace of $\R^n$ spanned by the ordered basis \[\{u_1,u_2,u_3,u_4\}:=\{e_1,e_2,e_{2k_1+1},e_{2k_1+2}\}.\] Choose $c,d\in \R$ with either $c\ne 0$ or $d\ne 0$. Let
\[w_1=cu_1+du_3 \quad \mbox{and} \quad w_2=cu_2+du_4.\] Then $w_1$ and $w_2$ are linearly independent, and the space spanned by $w_1$ and $w_2$ is a 2-dimensional subspace $W$ of $U$. We claim that $W$ is $T$-invariant. To see this, let $w\in W$ and write \[w=r_1w_1+r_2w_2=r_1cu_1+r_2cu_2+r_1du_3+r_2du_4,\]
where $r_1, r_2\in \R$.
It is then straightforward to show that
\[T(w)=(ar_1-br_2)w_1+(br_1+ar_2)w_2\in W.\] Varying the choices of $c$ and $d$ yields infinitely many such distinct subspaces $W$ of $U$ and the proof is complete.
\end{proof}
Since it is our goal to enumerate the $T$-invariant subspaces of $\R^n$, we assume, for every zero $\alpha$ of $c_T$ described in Theorem \ref{Thm:Inf}, that $T$ has exactly one corresponding Jordan block. One ramification of this assumption is that $c_T=m_T$. In addition, if $c_T=\prod_{i=1}^t f_i^{k_i}$ and $f_i(\alpha)=0$, then the size of the Jordan block corresponding to $\alpha$ is $k_i$ or $2k_i$, depending on whether $\alpha$ is real or non-real, respectively. Since the Jordan blocks of $T$ are independent \cite{HK1971}, we focus on the $T$-invariant subspaces of a single block. Then we can piece together the various $T$-invariant subspaces from each block via direct sums to determine all $T$-invariant subspaces.

 \begin{lem}\label{Lem:SE and RU}
 Let $T\in \L(\R^n)$, and assume that $[T]=J$, where $J$ is a Jordan block.
 \begin{enumerate}
 \item \label{Part:SE} If $n=k$ and $J=J_{\lambda, k}$, then for each integer $r\ge 0$, $E_{\lambda}^{(r)}$ is $T$-invariant. Also,
  \[E_{\lambda}^{(0)}<E_{\lambda}^{(1)}<\cdots <E_{\lambda}^{(k-1)}<E_{\lambda}^{(k)}=E_{\lambda}^{(k+1)}=\cdots\]
  with $\dim(E_{\lambda}^{(r)})=r$ for all $0\le r\le k$.
  \item \label{Part:RU} If $n=2k$ and $J=J_{\zeta, k}^*$, then
 for each integer $r\ge 0$, $U_{\zeta}^{(r)}$ is $T$-invariant. Also,
  \[U_{\zeta}^{(0)}<U_{\zeta}^{(1)}<\cdots <U_{\zeta}^{(k-1)}<U_{\zeta}^{(k)}=U_{\zeta}^{(k+1)}=\cdots\]
  with $\dim(U_{\zeta}^{(r)})=2r$ for all $0\le r\le k$.
  \end{enumerate}
 \end{lem}
 \begin{proof}
 To prove \emph{\ref{Part:SE}.}, let $v\in E_{\lambda}^{(r)}$. Since $J_{\lambda,k}$ commutes with $J_{\lambda,k}-\lambda I_{k}$, we have
 \[\left(J_{\lambda,k}-\lambda I_{k}\right)^{r}T(v)=T\left(J_{\lambda,k}-\lambda I_{k}\right)^{r}(v)=T(0)=0,\]
 so that $E_{\lambda}^{(r)}$ is $T$-invariant. The other statements in part \emph{\ref{Part:SE}.} follow from the fact that $\{e_1,e_2,\ldots, e_r\}$ is a basis for $E_{\lambda}^{(r)}$.

 For \emph{\ref{Part:RU}.}, as in the proof of part \emph{\ref{Part:SE}.}, $U_{\zeta}^{(r)}$ is $T$-invariant since $J_{\zeta,k}^*$ commutes with $J_{\zeta,k}^*-D_{k}$. The other statements in part \emph{\ref{Part:RU}.} follow from the fact that $\{e_1,e_2,\ldots, e_{2r}\}$ is a basis for $U_{\zeta}^{(r)}$.
 \end{proof}

\begin{thm}\label{Thm: T-invariant Single Block}
Let $T\in \L(\R^n)$, and assume that $[T]=J$, where $J$ is a Jordan block. If $n=k$ and $J=J_{\lambda, k}$, then the $T$-invariant subspaces of $\R^n$ are precisely $E_{\lambda}^{(r)}$, for $0\le r \le k$. If $n=2k$ and $J=J_{\zeta, k}^*$, then the $T$-invariant subspaces of $\R^n$ are precisely $U_{\zeta}^{(r)}$, for $0\le r \le k$.
\end{thm}
\begin{proof}
Let $W$ be a $T$-invariant subspace of $\R^n$. First suppose that $n=k$ and $J=J_{\lambda, k}$. Then, for any $w\in W$, we have that
 \[(T-\lambda I_k)w=T(w)-\lambda w\in W.\] Thus $W$ contains the set of vectors
 \[A_w=\{(T-\lambda I_k)^iw\mid i=0,1,2,\ldots\}.\]
  Choose $w=(a_1,a_2,\ldots ,a_r,0,\ldots,0)\in W$ such that $a_r\ne 0$ and no other vector in $W$ has a nonzero component in any location to the right of $r$. Then $\dim(W)\le r$ and
 \begin{align*}
 A_w&=\{(a_1,a_2,\ldots ,a_r,0,\ldots,0), (a_2,a_3,\ldots ,a_r,0,\ldots,0), (a_3,a_4,\ldots ,a_r,0,\ldots,0),\\
 & \quad \ldots ,(a_r,0,\ldots,0),(0,0,\ldots,0)\}.
 \end{align*}
 Since $a_r\ne 0$, it follows that $W$ contains the set $\{e_1,e_2,\ldots ,e_r\}$. Hence, $W=E_{\lambda}^{(r)}$ from part \emph{\ref{Part:SE}.} of Lemma \ref{Lem:SE and RU}.

 Now suppose that $n=2k$ and $J=J_{\zeta, k}^*$, where $\zeta=a+bi$ with $b>0$. Let $W$ be a $T$-invariant subspace of $\R^{2k}$. Let $w=(w_1,w_2,\ldots ,w_s,0\ldots,0)\in W$ such that $w_s\ne 0$ and no other vector in $W$ has a nonzero component in any location to the right of $s$. Then $s$ is even. For if $w_s\ne 0$ and $s$ is odd, then the component of $T(w)$ at location $s+1$ is $bw_s$, which is nonzero since $b\ne 0$, contradicting our choice of $w$. So let $s=2r$. Note that $W\subseteq U_{\zeta}^{(r)}$. We proceed by induction on $r$ to show that $W=U_{\zeta}^{(r)}$. Observe that
 \begin{align*}
 \widetilde{w}:=&\frac{\frac{T(w)-aw}{-b}+w_{2r-1}w}{1+w_{2r-1}^2}=(*,*,\ldots,*,1,0,\ldots ,0)\in W, \quad \mbox{and}\\
 \\
 \widehat{w}:=&w-w_{2r-1}\widetilde{w}=(*,*,\ldots ,*,0,1,0,\ldots ,0)\in W,
 \end{align*} where the $1$ is at location $2r-1$ in $\widetilde{w}$ and the $1$ is at location $2r$ in $\widehat{w}$. The case $r=1$ is then immediate. A key observation is that, for any $r\ge 2$ and $\widehat{w}=(*,*,\ldots ,*,x,y,0,1,0,\ldots ,0)$ where $x$ is in location $2r-3$, the determinant of the $4\times 4$ matrix whose rows are the four components of $\widehat{w}$, $T(\widehat{w})$, $T^2(\widehat{w})$ and $T^3(\widehat{w})$ in locations $2m-3, 2m-2, 2m-1$ and $2m$, is $4b^4\ne 0$. Thus, when $r=2$, we have $\dim(W)=4$, and hence $W=U_{\zeta}^{(2)}$. Assume by induction, for all $r\le m-1$, that $W=U_{\zeta}^{(r)}$ when the rightmost location of a nonzero component for any vector in $W$ is $2r$. Now suppose $r=m$, so that the rightmost location of a nonzero component of any vector in $W$ is $2m$. From our key observation,  the four vectors $\widehat{w}$, $T(\widehat{w})$, $T^2(\widehat{w})$ and $T^3(\widehat{w})$, where $\widehat{w}=(*,*,\ldots ,*,x,y,0,1,0,\ldots ,0)$ with $x$ in location $2m-3$, are linearly independent. Thus, $W$ contains an element of the form $(*,*,\ldots ,*,*,z,0,\ldots ,0)$, where $z\ne 0$ is in location $2m-2$. Let
 \[W_0:=\{w\in W \mid w=(w_1,w_2,\ldots ,w_{2m-2},0,\ldots ,0)\}.\]
 Clearly, $W_0$ is a subspace of $W$, and since $w_{2m-2}\ne 0$ for some $w\in W_0$, we have by induction that $W_0=U_{\zeta}^{(m-1)}$. Then since
 \[U_{\zeta}^{(m-1)}\subset W \subseteq U_{\zeta}^{(m)},\]
 we see that $\{e_1,e_2,\ldots ,e_{2m-2}\}\subseteq W$. Using this information and the fact that $\widetilde{w},\widehat{w}\in W$, we get that $\{e_{2m-1}, e_{2m}\}\subseteq W$ to conclude that $W=U_{\zeta}^{(m)}$, which completes the proof of the theorem.
 \end{proof}

%The following lemma is needed for the proof of the main result. \begin{lem}\label{Lem:ESP and Partitions}\end{lem}

In light of the previous results, we let $T \in \L(\R^n)$ with $c_T=m_T=\prod_{j=1}^{t}f_j^{k_j}$, where each $f_j$ is either a quadratic factor with non-real zeros, or a linear factor. We also
let \[[T]=\begin{bmatrix} J_1 & 0 &  \cdots & 0 & 0\\ 0 & J_2 &  \cdots & 0 & 0\\ \vdots & \vdots & \ddots & \vdots & \vdots \\ 0 & 0 & \cdots & J_{t-1} & 0\\ 0 & 0 & \cdots & 0 & J_t\end{bmatrix}=\bigoplus_{j=1}^t J_j\] be the real Jordan form of $T$,
where each $J_j$ corresponds to the factor $f_j^{k_j}$ in $c_T$, and that the real Jordan blocks appear first on the diagonal, reading from upper left to lower right, followed by the standard Jordan blocks. Note that if $J_j$ is a real Jordan block, then its size is $2k_j$, while if $J_j$ is a standard Jordan block, its size is $k_j$. Thus, if the last row in the last real Jordan block is $2r$, then $0\le r \le \left \lfloor \frac{n}{2} \right \rfloor$, and the last $s=n-2r$ rows of $[T]$ are dedicated to the standard Jordan blocks. We are now in a position to prove the main result.

\begin{thm}\label{Thm: Main Theorem}
 For a positive integer $r$, with $0\le r \le \left \lfloor \frac{n}{2} \right \rfloor$, let $s=n-2r$ and let $d_r(\mu,\theta)$ be the derived composition of the $\mu$-multipartition $\theta$, where $\mu =(r,s)\vDash n$ and $\theta=(\theta_1,\theta_2)$, with $\theta_1\vdash r$ and $\theta_2 \vdash s$. Then the exact set $\M_n$ of positive integers $m$ for which there exists $T \in \L(\R^n)$ such that $\R^n$ has exactly $m$ T-invariant subspaces is
\[\M_n=\left\{ \prod_{i=1}^{\ell\left( d_r(\mu,\theta)\right)}\left(d_r(\mu,\theta)_i+1\right)\, \mid \, 0\le r \le \left \lfloor \frac{n}{2} \right \rfloor\right\}.\]
\end{thm}
\begin{proof}
The real Jordan blocks occupy the first $2r$ rows of $[T]$. There could be a single block of size $2r$; or two blocks, consisting of one of size $2r-2$ and one of size 2, or one of size $2r-4$ and one of size 4, and so on. Since the order of the blocks is irrelevant, these possibilities can be described simply by the partitions of $r$, where a block of size $2k$ corresponds to a part in the partition of $r$ of size $k$. Each such block contains exactly $k+1$ nested $T$-invariant subspaces by Theorem \ref{Thm: T-invariant Single Block}. The analysis of the standard Jordan blocks in $[T]$ is similar, except there we use partitions of $s=n-2r$. Since the blocks are independent, any $T$-invariant subspace will be a direct sum of $T$-invariant subspaces originating from the individual blocks. Hence, for any two particular partitions $\theta_1\vdash r$ and $\theta_2\vdash s$, we can count the total number of $T$-invariant subspaces by considering the derived composition $d_r(\mu ,\theta)$, where $\mu$ is the composition $(r,s)$ and $\theta=(\theta_1,\theta_2)$ is a multipartition. Then, each part of $d_r(\mu ,\theta)$ corresponds to a single Jordan block in $[T]$ and the total number of $T$-invariant subspaces for a fixed $r$ is $\prod_{i=1}^{\ell\left( d_r(\mu,\theta)\right)}\left(d_r(\mu,\theta)_i+1\right)$. Letting $r$ range from 0 to $\left \lfloor \frac{n}{2} \right \rfloor$ exhausts all possibilities and completes the proof.
\end{proof}

We end this section with an example.
\begin{ex}
 We examine the case of $n=4$. We list the possibilities for $d_r(\mu,\theta)$ and the corresponding number $N_{r}$ of $T$-invariant subspaces for each value of $r$ with $0\le r \le \left \lfloor \frac{n}{2} \right \rfloor=2$. Recall that $s=n-2r$.
 \begin{itemize}
   \item $r=0$ and $s=4$.\\
   Thus $\mu=(0,4)$, and
   \[\theta_1\in \{(0)\}, \quad  \theta_2\in \{(4),(3,1),(2,2),(2,1,1),(1,1,1,1)\}.\]

 \begin{table}[h]
 \centering
       % Give a unique label
\begin{tabular}{cc}
%\hline\noalign{\smallskip}
$d_0(\mu,\theta)$ & $N_{0}$\\
\noalign{\smallskip}\hline\noalign{\smallskip}
 (0,4) & 5\\
     (0,3,1) & 8\\
     (0,2,2) & 9\\
     (0,2,1,1) & 12\\
     (0,1,1,1,1) & 16\\
\noalign{\smallskip}\hline\\
\end{tabular}
\caption{Values of $d_0(\mu,\theta)$ and $N_{0}$}
\label{Table:1}
\end{table}

   \item $r=1$ and $s=2$.\\
   Thus $\mu=(1,2)$, and
   \[\theta_1\in \{(1)\}, \quad \theta_2\in \{(2),(1,1)\}.\]
   \begin{table}[h]
\centering
      % Give a unique label
\begin{tabular}{cc}
%\hline\noalign{\smallskip}
$d_1(\mu,\theta)$ & $N_{1}$\\
\noalign{\smallskip}\hline\noalign{\smallskip}
 (1,2) & 6\\
     (1,1,1) & 8\\
\noalign{\smallskip}\hline\\
\end{tabular}
\caption{Values of $d_1(\mu,\theta)$ and $N_{1}$}
\label{Table:2}
\end{table}

   \item $r=2$ and $s=0$.\\
   Thus $\mu=(2,0)$, and
   \[\theta_1\in \{(2),(1,1)\}, \quad  \theta_2\in \{(0)\}.\]

   \begin{table}[h]
\centering
      % Give a unique label
\begin{tabular}{cc}
%\hline\noalign{\smallskip}
$d_2(\mu,\theta)$ & $N_{2}$\\
\noalign{\smallskip}\hline\noalign{\smallskip}
(2,0) & 3\\
     (1,1,0) & 4\\
\noalign{\smallskip}\hline\\
\end{tabular}
\caption{Values of $d_2(\mu,\theta)$ and $N_{2}$}
\label{Table:3}
\end{table}

 \end{itemize}
 \newpage
Combining the information from Tables \ref{Table:1}, \ref{Table:2} and \ref{Table:3} gives \[\M_4=\{3,4,5,6,8,9,12,16\}.\]
\end{ex}

\section{Extending The Results}\label{Sec: Extension}

In this section we outline an approach, and indicate some of the difficulties therein, to extend the results of Section \ref{Sec: Results} to finite dimensional vector spaces $V$ over an arbitrary field $\F$. We point out that this problem has recently been solved in the case when $\F$ is finite \cite{HF2011}. We begin with a definition.
\begin{defn}[Generalized Jordan Block]\label{Def:GenJordanBlock}
Let $\F$ be a field, and let $f(x)\in \F[x]$ be a monic polynomial of degree $d$ that is irreducible over $\F$. The $kd\times kd$ \emph{generalized Jordan block corresponding to $f(x)$} is the $kd \times kd$ matrix
 \[J_{f(x),k}:=\begin{bmatrix} C(f(x)) & N & 0 & \cdots & 0 & 0\\ 0 & C(f(x)) & N & \cdots & 0 & 0\\ \vdots & \vdots & \vdots & \ddots & \vdots & \vdots\\ 0 & 0 & 0 & \cdots & C(f(x)) & N\\ 0 & 0 & 0 & \cdots & 0 & C(f(x))\end{bmatrix},\]
 where $C(f(x))$ is the companion matrix of $f(x)$, and $N$ is the $d\times d$ matrix with a 1 in location $(1,d)$ and zeros everywhere else.
\end{defn}
\begin{rem}
  The idea of a generalized Jordan block can be extended further to allow for the possibility that $f(x)$ is not irreducible over $\F$, but such a generalization is not needed here \cite{AW1992}.
\end{rem}
\begin{thm}[General Jordan Canonical Form]\label{Thm:GJCF}
Let $V$ be a finite-dimensional vector space over a field $\F$ and let $T\in \L(V)$.
 Then there is an ordered basis $\B$ of $V$ such that
\[[T]_{\B}=\bigoplus_{j=1}^t J_j,\]
where each $J_j$ is a generalized Jordan block, and the matrix $[T]_{\B}$ is unique up to the order of the blocks.
\end{thm}

A careful analysis of the invariant subspaces of the individual generalized Jordan blocks, combined with Theorem \ref{Thm:GJCF}, should yield a theorem analogous to Theorem \ref{Thm: Main Theorem} in this situation. One annoyance is that $\F$ might not have the property that a finite extension of $\F$ is algebraically closed, as is the case of $\R$. In other words, there might exist irreducible polynomials over $\F$ of arbitrarily large degree. However, this concern could be handled on somewhat of a case-by-case basis. Since $\dim_{\F}(V)$ is finite, the possibilities for the sizes of the generalized Jordan blocks would be limited as well. That is, the largest degree of an irreducible factor of $c_T$ would be bounded above by $\dim_{\F}(V)$. One would still have to consider the possibility of blocks of various sizes corresponding to irreducible polynomials whose degrees range from 1 to $\dim_{\F}(V)$.
%\section{Acknowledgements}



%% References

\bibliographystyle{plain}
%\bibliographystyle{model1b-num-names}
%\bibliography{LJbib(12-29-11)}

%% not want to use model1b-num-names.bst.





% \end{thebibliography}

\end{document}